\documentclass[11pt]{article}
\PassOptionsToPackage{obeyspaces}{url}
\usepackage{amsfonts, amsmath, amssymb}
\usepackage[all]{xy}
\usepackage[normalem]{ulem}
\usepackage[hidelinks]{hyperref}
\usepackage{bookmark}
\bookmarksetup{open,numbered,depth=5} 

\usepackage{amsthm}
\usepackage{thm-restate}

\usepackage{breakurl}
\usepackage{tikz}
\usetikzlibrary{calc}
\usetikzlibrary{decorations.pathmorphing}
\usepackage{algpseudocode}

\usepackage{etoolbox} 
\patchcmd{\thebibliography}{\leftmargin\labelwidth}{\leftmargin\labelwidth\addtolength\itemsep{-0.1\baselineskip}}{}{}

\usepackage{mathtools}

\usepackage{epstopdf}

\usepackage{paralist}

\usepackage[justification=raggedright]{caption}
\author{Michael Harrison\thanks{Institute for Advanced Study, Princeton, NJ 08540, USA\@.  Supported by the National Science Foundation through Award No. 1926686.} \and R. Amzi Jeffs\thanks{Department of Mathematical Sciences, Carnegie Mellon University, Pittsburgh, PA 15213, USA\@. Supported by the National Science Foundation through Award No. 2103206.}}

\title{Quantitative upper bounds on the\\Gromov--Hausdorff distance between spheres}
\date{August 2024}

\usepackage[nameinlink]{cleveref}

\newtheorem{theorem}{Theorem}[section]
\newtheorem{lemma}[theorem]{Lemma}
\newtheorem{corollary}[theorem]{Corollary}
\newtheorem{proposition}[theorem]{Proposition}

\newtheorem{definition}[theorem]{Definition}

\theoremstyle{remark}
\newtheorem{example}[theorem]{Example}
\newtheorem{remark}[theorem]{Remark}
\newtheorem{question}[theorem]{Question}

\newcommand*{\eqdef}{\stackrel{\mbox{\normalfont\tiny def}}{=}}  
\newcommand*{\R}{\mathbb{R}}                                     
\newcommand*{\Z}{\mathbb{Z}}                                     
\newcommand*{\RP}{\mathbb{R}\mathrm{P}}                                     
\newcommand*{\Pc}{\mathcal G}                                     
\newcommand*{\Q}{\mathcal{F}}                                     
\DeclareMathOperator{\diam}{diam}                                 
\DeclareMathOperator{\vdiam}{vdiam}                                 
\DeclareMathOperator{\dis}{dis}                               
\DeclareMathOperator{\sep}{sep}                               
\newcommand{\gh}{\mathrm{GH}}                               
\newcommand{\HH}{\mathrm{H}}

\newcommand{\Rk}{\mathcal R_k}
\newcommand{\cdott}{\hspace{0.1em}}

\begin{document}
\maketitle
\begin{abstract}
The Gromov--Hausdorff distance between two metric spaces measures how far the spaces are from being isometric. 
It has played an important and longstanding role in geometry and shape comparison.
More recently, it has been discovered that the Gromov--Hausdorff distance between unit spheres equipped with the geodesic metric has important connections to Borsuk--Ulam theorems and Vietoris--Rips complexes. 

We develop a discrete framework for obtaining upper bounds on  the Gromov--Hausdorff distance between spheres, and provide the first quantitative bounds that apply to spheres of all possible pairs of dimensions. 
As a special case, we determine the exact Gromov--Hausdorff distance between a circle and any higher-dimensional sphere, and determine the precise asymptotic behavior of the distance from the 2-sphere to the $k$-sphere up to constants.
\end{abstract}
\section{Introduction}\label{sec:intro}
The Gromov--Hausdorff distance $d_\gh(X,Y)$ between two compact metric spaces $X$ and $Y$ captures how closely the spaces can be aligned with one another---in particular, $d_\gh(X,Y) = 0$ if and only if $X$ and $Y$ are isometric \cite{BBI2001, petersen2006}.
The Gromov--Hausdorff distance was first defined in 1975 by Edwards~\cite{Edwards75}, and then rediscovered in 1981 by Gromov~\cite{Gromov81}, and for decades it has played an important role in metric geometry (see e.g. \cite{cheeger1997, colding1996}).
In more recent years, it has proved to be a natural tool in the context of data analysis and shape matching \cite{BBK2008,memoli2007,MS05, CM2010}.

Precisely determining the Gromov--Hausdorff distance between given metric spaces is computationally challenging \cite{schmiedl2017, memoli2012, AFNSW2018}. 
In particular, computing the Gromov--Hausdorff distance between finite metric spaces is NP-hard \cite{memoli2007}.
Exact values of the Gromov--Hausdorff distance in special cases have been obtained  only very recently---for example, between certain discrete metric spaces~\cite{IT19}, between an interval and a circle~\cite{JT22}, between vertex sets of regular polygons~\cite{Talipov22}, \cite[Appendix C]{LMS22}.

The first nontrivial bounds on the Gromov--Hausdorff distance between unit spheres equipped with the geodesic metric were obtained in recent work of Lim, M\'emoli, and Smith~\cite{LMS22}, who provided lower bounds using novel connections to Borsuk--Ulam type theorems, and upper bounds via explicit constructive techniques.
They computed the exact Gromov--Hausdorff distance between $S^n$ and $S^k$ for $n=0$, $k=\infty$, and $(n,k)\in \{(1,2), (1,3), (2,3)\}$.
They also provided quantitative upper bounds in the case $k = n+1$, and for general $n$ and $k$ they showed that $d_\gh(S^n, S^k)$ is strictly less than the trivial upper bound of $\tfrac{\pi}{2}$, which follows from the fact that $S^n$ and $S^k$ have diameter $\pi$.

The aforementioned lower bounds were subsequently improved in a large polymath project~\cite{dghpolymath}, which included the present authors.
This project built on insights of Adams, Bush, and Frick~\cite{ABF21} and Lim, M\'emoli, and Smith~\cite{LMS22} to demonstrate quantitative relationships between the Gromov--Hausdorff distance $d_\gh(S^n, S^k)$, Borsuk--Ulam type theorems concerning functions $S^k\to S^n$, and the topology of Vietoris--Rips complexes over $S^n$ and $S^k$.
These quantitative connections further motivate the question of determining the exact value of the Gromov--Hausdorff distance between spheres. 
Explicit upper bounds on $d_\gh(S^n, S^k)$ for general $n$ and $k$ have remained elusive, and our work addresses this scarcity by computing the exact value of $d_\gh(S^1,S^k)$ for all $k$ and giving quantitative upper bounds on $d_\gh(S^n,S^k)$ for all $n$ and $k$.

\paragraph{Quantitative results.}

For reasons of notational convenience, we will work with the quantity $2\cdott d_\gh(S^n, S^k)$.
The statements of \Cref{thm:1tok,thm:ntok}  contain two equivalent formulations: first an upper bound on $2\cdott d_{\gh}(S^n,S^k)$, and then a lower bound on the gap between $2\cdott d_\gh(S^n,S^k)$ and the elementary upper bound of $\pi$.
When $n=1$, our methods are strong enough to meet the existing lower bounds, and we obtain the following theorem.

\begin{restatable}{theorem}{1tok}\label{thm:1tok}
Let $\ell\ge 1$ be any integer. 
Then
\begin{align*}
2\cdott d_\gh(S^1, S^{2\ell}) = \hspace{1.7em} 2\cdott d_\gh(S^1, S^{2\ell+1}) &= \frac{2\pi \ell}{2\ell+1}\\
\text{or equivalently}\quad\quad  \pi - 2\cdott d_\gh(S^1, S^{2\ell}) = \pi -  2\cdott d_\gh(S^1, S^{2\ell+1}) &= \frac{\pi}{2\ell+1}.
\end{align*}
\end{restatable}

The lower bounds for $2\cdott d_\gh(S^1, S^{2\ell})$ and $2\cdott d_\gh(S^1, S^{2\ell+1})$ were achieved in~\cite{dghpolymath},
where it is shown that, for all $n$ and $k$, strong bounds for $2\cdott d_\gh(S^n, S^k)$ can be obtained in terms of the equivariant topology of Vietoris--Rips complexes. \Cref{thm:1tok} shows that these bounds are sharp in infinitely many cases, answering half of \cite[Question 8.1]{dghpolymath} in the affirmative.
It could be the case that these topological lower bounds for $2\cdott d_\gh(S^n,S^k)$ are sharp for \emph{every} $n$ and $k$, and \Cref{thm:1tok} provides the strongest evidence to date in this direction.

The upper bound $2\cdott d_\gh(S^1, S^{2\ell}) \le \frac{2\pi\ell}{2\ell+1}$ can be obtained relatively quickly with our techniques---in particular, it is a special case of \Cref{thm:ntok} below, with the choice $n=1$ and $k = 2\ell$.
The Gromov--Hausdorff distance between a circle and an odd-dimensional sphere will require a much more careful analysis, which we carry out in \Cref{sec:1tok-kodd}.

For general $n$ and $k$, we prove the following result.
\begin{restatable}{theorem}{ntok}\label{thm:ntok}
For every $1 \le n < k <\infty$, we have 
\[
2\cdott d_\gh(S^n, S^k) \le \frac{\pi k}{k+1}\quad\quad \text{or equivalently} \quad \quad \pi - 2\cdott d_\gh(S^n, S^k) \ge \frac{\pi}{k+1}.
\]
\end{restatable}
\noindent We will establish the $n=1$ case of \Cref{thm:ntok} in  \Cref{sec:1tok-initial}. 
The general case is treated in \Cref{sec:general}.

\paragraph{Packings, coverings, and asymptotic results.}
The bound of \Cref{thm:ntok} has the advantage that it is concrete and simple to state.
However, we are able to prove a more granular set of results which gives an asymptotic improvement on \Cref{thm:ntok} for fixed $n\ge 2$, and gives a sharp (up to constants) asymptotic description of $2\cdott d_\gh(S^2,S^k)$.
 To discuss these results, we first introduce some definitions and notation.

We will consider projective space $\RP^n$, endowed with the quotient metric induced by the geodesic metric on $S^n$.
We denote this metric by $d_{\RP^n}$.
Below, we will make repeated use of the following two parameters:
\begin{align*}
p_m(\RP^n) &\eqdef \sup \{\varepsilon > 0 \mid \text{$\exists x_1,\ldots, x_m$ in $\RP^n$ so that $d_{\RP^n}(x_i,x_j) \ge \varepsilon$ for all $i\neq j$ }\},\\
c_m(\RP^n) &\eqdef\, \inf \{\varepsilon > 0 \mid \text{$\exists x_1,\ldots, x_m$ in $\RP^n$ so that for all $x\in \RP^n$ there is $i$ with $d_{\RP^n}(x, x_i)\le \varepsilon$ }\}.
\end{align*}
The parameter $p_m(\RP^n)$ is the minimum distance between distinct points in an optimal packing of $m$ points in $\RP^n$ (i.e. a projective code). 
The parameter $c_m(\RP^n)$ is the smallest radius needed to cover $\RP^n$ by $m$-many metric balls. 
By considering an optimal packing with the additional property that as few points as possible have pairwise distance exactly $p_m(\RP^n)$, one can see that $c_m(\RP^n) \le p_m(\RP^n)$. 
A standard argument using the fact that the volume of a ball in $\RP^n$ is proportional to the $n$-th power of its radius shows that, for fixed $n$, $p_m(\RP^n)$ and $c_m(\RP^n)$ are bounded above and below by constant multiples of $\frac{1}{\sqrt[n]{m}}$. 
With these parameters in hand, we proceed to our results.

\begin{restatable}{theorem}{packing}\label{thm:packing}
For all $2\le n < k < \infty$, we have \[
2\cdott d_{\gh}(S^n, S^k) \le \max \left \{\arccos\left(\frac{-(k-1)}{k+1}\right) ,\,\, \pi - p_{k+1}(\RP^n),\,\, 2p_{k+1}(\RP^n) \right\}.
\]
\end{restatable}

Using \Cref{thm:packing}, we obtain an asymptotic result describing the gap between $2\cdott d_{\gh}(S^n, S^k)$ and $\pi$ for fixed $n$.
Below, $\Omega\left(\frac{1}{\sqrt{k}}\right)$ denotes a function that is bounded below by a positive multiple of $\frac{1}{\sqrt{k}}$.

\begin{restatable}{corollary}{asymptoticcor}\label{cor:asymptotic}
For fixed $n\ge 2$, we have \[
\pi - 2\cdott d_{\gh}(S^n, S^k) =  \Omega\left(\frac{1}{\sqrt{k}}\right).
\]
\end{restatable}

Previous work (see \cite[Theorem 5.3]{dghpolymath}) has shown that $2\cdott d_\gh(S^n, S^k) \ge \pi - 2c_{k}(\RP^n)$.
In particular $\pi - 2\cdott d_\gh(S^n,S^k) \le \Theta\left(\tfrac{1}{\sqrt[n]{k}}\right)$, where $\Theta\left(\tfrac{1}{\sqrt[n]{k}}\right)$ denotes a positive function that is bounded above and below by constant multiples of $\tfrac{1}{\sqrt[n]{k}}$.
For $n=2$, this matches (up to constants) the upper bound of \Cref{cor:asymptotic}, allowing us to exactly determine the asymptotic behavior of $\pi - 2\cdott d_\gh(S^2, S^k)$.
\begin{corollary}\label{cor:asymptotic2}
The Gromov--Hausdorff distance between $S^2$ and $S^k$ satisfies \[
\pi - 2\cdott d_{\gh}(S^2, S^k) = \Theta\left(\frac{1}{\sqrt{k}}\right).
\]
\end{corollary}

We conjecture that in general the $n$-th root bound is the correct answer for the asymptotic behavior of $\pi - 2\cdott d_\gh(S^n,S^k)$ when $n$ is fixed.
Note that \Cref{thm:1tok} and \Cref{cor:asymptotic2} prove this conjecture for $n=1$ and $n=2$ respectively.

\begin{restatable}{conjecture}{asymptoticconj}\label{conj:asymptotic}
For fixed $n\ge 1$, we have 
\[
\pi - 2\cdott d_{\gh}(S^n, S^k) = \Theta\left(\frac{1}{\sqrt[n]{k}}\right).
\]
\end{restatable}

\paragraph{Our methods.}

To prove upper bounds on $2\cdott d_\gh(S^n, S^k)$, it suffices to construct (possibly discontinuous) correspondences between $S^n$ and $S^k$ which do not distort the metric too much (see \Cref{sec:background}). 
 Lim, M\'emoli, and Smith~\cite{LMS22} constructed optimal correspondences between low-dimensional spheres by decomposing the larger sphere into discrete chunks, which were then collapsed to points in the lower-dimensional sphere. 
 They used a similar approach to build correspondences between $S^{n+1}$ and $S^n$ for all $n$, a construction which was further developed in \cite[Section 6]{dghpolymath}.
 Continuing in a similar vein, our approach provides a clean framework for constructing correspondences between $S^n$ and $S^k$ for any $n$ and $k$, and it bounds the resulting distortion (see \Cref{thm:induced-correspondence}).
 Despite the thematic through-line in the development of these methods, the particular correspondences that we construct are wholly new---existing correspondences in the literature are not a special case of them.

More specifically, our approach uses finite centrally symmetric point sets in $S^n$ and $S^k$ to break the spheres into discrete chunks, which are then collapsed into the opposite sphere in such a way that the distortion is quantifiable.
It is possible that our bounds could be improved in some cases with different choices of centrally symmetric point sets, but we suspect
that determining the exact value of $d_{\gh}(S^n,S^k)$ for $2\le n < k$ will require constructions that are not discrete in this way.
In fact, the completely discrete approach is already insufficient to prove the odd-dimensional case of \Cref{thm:1tok}; see \Cref{rem:kodd}.
Our proof of the odd-dimensional case in \Cref{sec:1tok-kodd} necessarily avoids collapsing full-dimensional chunks of either sphere to a point, but we maintain enough discrete structure that we are able to quantify the distortion. 
This approach builds on the ``embedding projection correspondences'' of M\'emoli and Smith~\cite{memolismith24}, and such an approach may be fruitful for general $n$ and $k$.

\section{Background}\label{sec:background}

The Gromov--Hausdorff distance $d_\gh(X, Y)$ between metric spaces $X$ and $Y$ is the infimum over all isometric embeddings of $X$ and $Y$ into a larger metric space, of the Hausdorff distance between their images.
In other words, $d_\gh(X,Y)$ captures how closely $X$ and $Y$ can be ``aligned" in some ambient space. 
Below we give an equivalent definition of the Gromov--Hausdorff distance due to Kalton and Ostrovskii~\cite{KO99}, which will prove convenient for our purposes. 

We first require some additional terminology.
A \emph{correspondence} between two sets $X$ and $Y$ is a relation $R\subseteq X\times Y$ with the property that $\pi_X(R) = X$ and $\pi_Y(R) = Y$; that is, the coordinate projections of $R$ to $X$ and $Y$ are surjective. Equivalently, a correspondence relates every point of $X$ to a point of $Y$, and vice versa.
If $X$ and $Y$ are metric spaces, the \emph{distortion} of a correspondence $R$ is \[
\dis(R) \eqdef \sup_{(x,y), (x',y')\in R} \big|d_X(x,x') - d_Y(y,y')\big|.
\]
Above, $d_X$ and $d_Y$ denote the respective metrics on $X$ and $Y$.
Informally, $\dis(R)$ is small if $R$ does not ``pull apart" or ``push together" points too much. 
Finally, the \emph{Gromov--Hausdorff distance} between compact metric spaces $X$ and $Y$ is given by the relationship \begin{equation}\label{eq:dgh}
2\cdott d_\gh(X,Y) \eqdef \inf_{R} \dis(R)
\end{equation}
where the infimum is taken over all correspondences $R\subseteq X\times Y$. 
 
 Our metric spaces of interest are unit spheres $S^n \eqdef \{x\in \R^{n+1} \mid \langle x,x\rangle = 1\}$, equipped with the \emph{geodesic metric}\[
 d_{S^n}(x,x') \eqdef \arccos(\langle x, x'\rangle). 
 \]
We will always consider spheres of positive finite dimension, as the cases of $S^0$ and $S^\infty$ have been fully treated by previous work~\cite{LMS22}. 
With the geodesic metric, $S^n$ has diameter equal to $\pi$, and a quick consequence is that $2\cdott d_\gh(S^n, S^k)$ is at most $\pi$.
As previously mentioned, Lim, M\'emoli, and Smith~\cite{LMS22} showed that in fact  $2\cdott d_\gh(S^n, S^k)$ is strictly less than $\pi$ when $1\le n < k < \infty$. 

We say that a set $P\subseteq S^n$ is \emph{antipodal} or \emph{centrally symmetric} if $x\in P$ implies $-x\in P$. 
We often write finite antipodal sets as $P = \{\pm p_1, \ldots, \pm p_m\}$. 
For a metric space $(X, d_X)$, the \emph{separation} of a finite set $P\subseteq X$ is the minimum distance between distinct points in the set, that is \[
\sep_X(P) \eqdef \min_{p\neq p'}\{d_{X}(p,p')\}
\]
where the minimum above is over all $p,p'\in P$ with $p\neq p'$. 
Note that if $P\subseteq S^n$ is an antipodal set containing more than two points, then $\sep_{S^n}(P)\le\frac{\pi}{2}$. 

Each finite subset $P = \{p_1,\ldots, p_m\}$ of a metric space $(X, d_X)$ can be associated to its \emph{Voronoi diagram}  $\{X_1, \ldots, X_m\}$ where $X_i\eqdef\{x\in X\mid d_X(x,p_i)\le d_X(x, p_j) \text{ for all $j\neq i$}\}$. 
We call $X_i$ the \emph{Voronoi cell} associated to $p_i$. 
If $P = \{p_1,\ldots, p_m\}$ is a finite subset of a metric space $(X, d_X)$ with Voronoi diagram $\{X_1, \ldots, X_m\}$, then the \emph{Voronoi diameter} of $P$ is \[
\vdiam_X(P) \eqdef \max_{i\in[m]}\big\{\diam_{X}(X_i)\big \}.
\]
Note that if $P\subseteq S^n$ is antipodal, then the Voronoi diagram of $P$ is centrally symmetric.
Also note that $\vdiam_X(P) \le 2\cdott d_\HH(X, P)$ where $d_\HH$ denotes Hausdorff distance.

\section{Correspondences induced by antipodal sets}

First, we informally describe our construction.
Choose finite antipodal sets in $S^n$ and $S^k$ with equal size. 
Each of these sets decomposes $S^n$ and $S^k$ into Voronoi cells.
Since there are equally many cells and points in each sphere, we may collapse cells in $S^n$ to points in $S^k$, and vice versa, obtaining a correspondence (see \Cref{fig:induced} below). 
Provided that the cells are not too large and the points are sufficiently separated, such correspondences will have distortion bounded away from $\pi$. 
The following theorem formalizes and quantifies this technique. 

\begin{theorem}\label{thm:induced-correspondence}
Let $P = \{\pm p_1, \ldots, \pm p_m\}\subseteq S^n$ and $Q=\{\pm q_1,\ldots, \pm q_m\}\subseteq S^k$ be finite antipodal sets of equal size.
Let $\{\pm F_1, \ldots, \pm F_m\}$ and $\{\pm G_1,\ldots, \pm G_m\}$ be the Voronoi diagrams in $S^n$ and $S^k$ induced by $P$ and $Q$ respectively.
Define the correspondence \[
R_{P, Q} \eqdef  \big\{\pm (p_i, q) \,\big\vert\, i\in[m]\text{ and }q\in G_i\big\} \cup \big\{\pm (p, q_i)\,\big\vert\, i\in[m]\text{ and }p\in F_i\big\}.
\]
The distortion of $R_{P, Q}$ is at most the maximum of the following four quantities: \begin{align*}
\vdiam_{S^n}(P) &\qquad\pi-\sep_{S^n}(P)\\
\vdiam_{S^k}(Q) &\qquad \pi-\sep_{S^k}(Q).
\end{align*}
\end{theorem}
\begin{figure}[h]
    \[
    \includegraphics{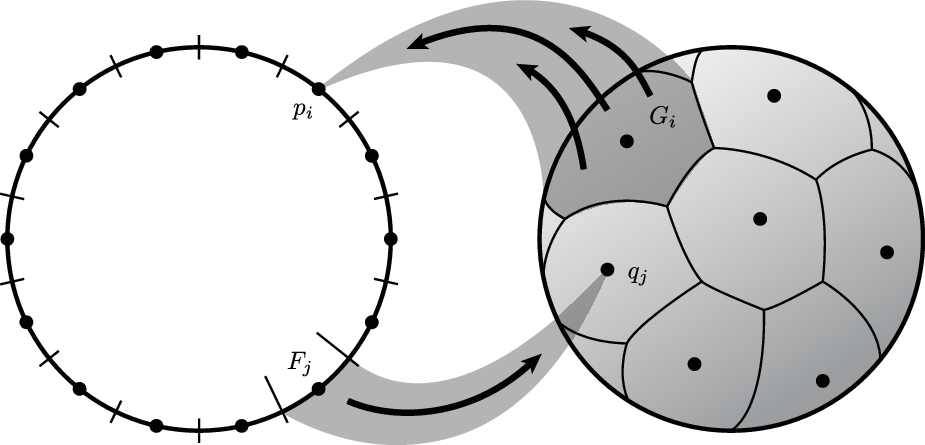}
    \]
    \caption{A sketch of the correspondence $R_{P,Q}$ of \Cref{thm:induced-correspondence}. 
    The cell $G_i\subseteq S^k$ collapses to the point $p_i\in S^n$, and the cell $F_j\subseteq S^n$ collapses to the point $q_j\in S^k$.}
    \label{fig:induced}
\end{figure}
\begin{proof}
First note that $R_{P,Q}$ is indeed a correspondence, as every point in $S^n$ appears in some Voronoi cell $\pm F_i$, and every point in $S^k$ appears in some Voronoi cell $\pm G_i$.
It remains to bound the quantity\begin{equation}\tag{$\dagger$}\label{eq:distortion}
\sup_{(x,y), (x',y')\in R_{P,Q}} \left|d_{S^n}(x,x') - d_{S^k}(y,y')\right|.
\end{equation}
We achieve this by considering a variety of cases, based on the definition of $R_{P,Q}$. 
Many of these cases are symmetric, thanks to the symmetry in the definition of $R_{P,Q}$, but we enumerate all of them for completeness.
To see that these nine cases cover all possibilities, note that if $(x,y)\in R_{P,Q}$ then $x\in P$ or $y\in Q$ (and possibly both are true). \smallskip

\noindent \textbf{Case 1:} $x = x' = \pm p_i$. \\
We have $y,y'\in G_i$ or $y,y'\in -G_i$. 
Thus (\ref{eq:distortion}) is equal to $d_{S^k}(y,y')$, which is at most $\vdiam_{S^k}(Q)$. \smallskip

\noindent \textbf{Case 2:} $x = -x' = \pm p_i$.\\
In this case $d_{S^n}(x,x') = \pi$, so (\ref{eq:distortion}) is equal to $\pi-d_{S^k}(y,y')$.
We have $y\in G_i$ and $y'\in -G_i$. 
By central symmetry of the Voronoi diagram in $S^k$, we have $-y'\in G_i$, and so \[
d_{S^k}(y,y') = \pi-d_{S^k}(y,-y')\ge \pi - \diam_{S^k}(G_i) \ge \pi-\vdiam_{S^k}(Q).\]
Plugging in and simplifying, we see (\ref{eq:distortion}) is again at most $\vdiam_{S^k}(Q)$. 
 \smallskip

\noindent \textbf{Case 3:} $x = \pm p_i$ and $x'= \pm p_j$ for $i\neq j$. \\
Here we have $\sep_{S^n}(P) \le d_{S^n}(x,x') \le \pi-\sep_{S^n}(P)$ since $x$ and $x'$ are distinct, non-antipodal points in $P$.  
From this, we see that  (\ref{eq:distortion}) is at most $\pi-\sep_{S^n}(P)$. 
 \smallskip

\noindent \textbf{Case 4:} $x = p_i$ and $y' = q_i$, or $x = -p_i$ and $y' = -q_i$.\\
By central symmetry, it suffices to consider the case $x = p_i$ and $y' = q_i$. 
We have $y\in G_i$ and $x'\in F_i$. 
Since $P$ and $Q$ are antipodal, we have $d_{S^n}(x,x') \le \pi/2$ and $d_{S^k}(y,y')\le \pi/2$.
Thus (\ref{eq:distortion}) is at most $\pi/2$, which is in turn at most $\pi-\sep_{S^n}(P)$ since $\sep_{S^n}(P)\le \pi/2$. 
\smallskip

\noindent \textbf{Case 5:} $x = p_i$ and $y' = -q_i$, or $x=-p_i$ and $y' = q_i$.\\
By central symmetry, it suffices to consider the case $x = p_i$ and $y' = -q_i$.
Here we have  $x'\in -F_i$ and $y\in G_i$. 
Using central symmetry of $P$ and $Q$, the former implies $d_{S^n}(p_i,x') \ge \pi/2$, and the latter implies $d_{S^k}(y, -q_i) \ge \pi/2$. 
These inequalities imply that (\ref{eq:distortion}) is at most $\pi/2$, which is in turn at most $\pi-\sep_{S^n}(P)$ since $\sep_{S^n}(P)\le \pi/2$.
\smallskip

\noindent \textbf{Case 6:} $x = \pm p_i$ and $y' = \pm q_j$ for $i\neq j$. \\
Here $x'$ lies in $\pm F_j$, and $y$ lies in $\pm G_i$. 
We claim that\[
\sep_{S^n}(P)/2 \le d_{S^n}(x,x') \le \pi-\sep_{S^n}(P)/2 \text{ and } \sep_{S^k}(Q)/2 \le d_{S^k}(y,y') \le \pi-\sep_{S^k}(Q)/2. 
\]
To see this, consider the case $x'\in F_j$ and $x=p_i$. Note by triangle inequality that \[
\sep_{S^n}(P) \le d_{S^n}(p_i, p_j) \le d_{S^n}(p_i, x') + d_{S^n}(x', p_j) \le 2d_{S^n}(p_i, x')
\]
where the last inequality follows from the fact that $x'\in F_j$.
This gives $\sep_{S^n}(P)/2 \le d_{S^n}(x,x')$, and the remaining cases follow by an analogous triangle inequality argument. 
Now, using the inequalities above we find that (\ref{eq:distortion}) is at most $\pi-(\sep_{S^n}(P) + \sep_{S^k}(Q))/2$, which is in turn bounded above by the maximum of $\pi - \sep_{S^n}(P)$ and $\pi-\sep_{S^k}(Q)$.
\smallskip

\noindent \textbf{Case 7:} $y = \pm q_i$ and $y' = \pm q_j$ for $i \neq j$.\\
Symmetric to Case 3. 
We have $\sep_{S^k}(Q) \le d_{S^n}(y,y') \le \pi-\sep_{S^k}(Q)$ since $y$ and $y'$ are distinct, non-antipodal points in $Q$.  
From this, we see that  (\ref{eq:distortion}) is at most $\pi-\sep_{S^k}(Q)$. 

\smallskip

\noindent \textbf{Case 8:} $y = -y' = \pm q_i$.\\
Symmetric to Case 2. 
Since $d_{S^k}(y,y') = \pi$, (\ref{eq:distortion}) is equal to $\pi-d_{S^n}(x,x')$.
However, $x\in F_i$ and $x'\in -F_i$, which implies $d_{S^n}(x,x') \ge \pi-\vdiam_{S^n}(P)$. 
Thus we find (\ref{eq:distortion}) is at most $\vdiam_{S^n}(P)$. 
\smallskip

\noindent \textbf{Case 9:} $y = y' = \pm q_i$. \\
Symmetric to Case 1. 
Here we have $x,x'\in F_i$ or $x,x'\in -F_i$, and so (\ref{eq:distortion}) is equal to $d_{S^n}(x,x')$, which is at most $\diam_{S^n}(F_i) \le \vdiam_{S^n}(P)$. 
\end{proof}

\begin{remark}
One may also wish to construct correspondences $R_{P,Q}$ using Voronoi diagrams of non-antipodal point sets $P$ and $Q$.
This can certainly be done, and a theorem similar to \Cref{thm:induced-correspondence} could be obtained, but the casework would be slightly different and more laborious.
The advantage of the antipodal symmetry in our theorem is that $R_{P,Q}$ will generally send antipodes to antipodes, and we need not worry that points which lie distance $\diam(S^n) = \diam(S^k) = \pi$ from another are ``pulled together'' by the correspondence. 
Without the antipodal assumption, several further quantities---such as the diameter of $P$ and $Q$---would need to be considered.
We have chosen the antipodal formulation since it simplifies our statements and constructions, and we are not aware of any cases where a non-antipodal argument would improve any of our bounds. 
This is consistent with the ``helmet trick" established by Lim, M\'emoli, and Smith~\cite{LMS22}, which shows that one may reduce to the case of antipode-preserving (or ``odd") functions between $S^n$ and $S^k$ when bounding distortion.
\end{remark}

\begin{remark}
The requirement that $P$ and $Q$ have the same size is one of the main technical obstacles in applying \Cref{thm:induced-correspondence}.
There are many natural antipodal sets in $S^n$, such as the vertices of the cross-polytope, the positive and negative copies of the vertices of a geodesic simplex, or the (normalized) vertices of a hypercube.
However, it is  usually not so straightforward to find a nice antipodal set of matching size in $S^k$.
\end{remark}

\section{From the 1-sphere to the $k$-sphere: an initial bound}\label{sec:1tok-initial}

To use \Cref{thm:induced-correspondence}, we must construct antipodal point sets in $S^n$ and $S^k$ with small Voronoi diameter and large separation, to the extent that this is possible.
Doing so in $S^1$ is straightforward: the best choice is simply to space points evenly around the circle.
In $S^k$ there are several natural choices, and one that will prove useful for us is the vertex set of the cross-polytope, i.e. the standard basis vectors and their negatives.
We start by characterizing the Voronoi diameter and separation of this set.

\begin{lemma}\label{lem:crosspolytope}
Let $Q = \{\pm e_1, \ldots, \pm e_{k+1}\}\subseteq S^k$ be the standard basis vectors and their negatives. 
Then for $k\ge 1$ we have\[
\vdiam_{S^k}(Q) = \arccos\left(\frac{-(k-1)}{k+1}\right) \qquad \text{ and } \qquad \sep_{S^k}(Q) = \frac{\pi}{2}. 
\]\end{lemma}
\begin{proof}
The equality $\sep_{S^k}(Q) = \frac{\pi}{2}$ is immediate since every pair of non-antipodal points in $Q$ are distance $\frac{\pi}{2}$ apart.
To obtain the Voronoi diameter, it suffices by symmetry to compute the diameter of the cell associated to $e_1$. This cell consists of all points in $S^k$ whose first coordinate is positive and maximum in absolute value among all the coordinates. 
In particular, every point in this cell has first coordinate at least as large as $\frac{1}{\sqrt{k+1}}$. 
Thus if $x$ and $x'$ lie in this cell, then
\begin{align*}
\langle x, x'\rangle
\ = \ 
\sum_{i=1}^{k+1} x_ix_i'
\ \ge \ 
\frac{1}{k+1} + \sum_{i=2}^{k+1} x_ix_i'
\ \ge \
\frac{1}{k+1} - \frac{k}{k+1}
\ = \ 
\frac{-(k-1)}{k+1}.
\end{align*}
Above, the last inequality follows by applying the Cauchy--Schwarz inequality to the truncated vectors obtained by deleting the first coordinates of $x$ and $x'$, noting that each truncated vector has norm at most $\sqrt{\tfrac{k}{k+1}}$. 
Taking the arccosine of both sides flips the direction of the inequality, and we obtain $d_{S^k}(x,x') \le \arccos\left(\tfrac{-(k-1)}{k+1}\right)$.
Moreover, equality is achieved when $x$ has all coordinates equal to $\tfrac{1}{\sqrt{k+1}}$ and $x'$ has first coordinate equal to $\tfrac{1}{\sqrt{k+1}}$ and all other coordinates equal to $\tfrac{-1}{\sqrt{k+1}}$.
\end{proof}

To analyze correspondences from $S^1$ to $S^k$, we first require a small technical lemma.

\begin{lemma}\label{lem:calc}
For $k\ge 3$, we have $\arccos\left(\frac{-(k-1)}{k+1}\right) \le \frac{(k-1)\pi}{k}$.
 \end{lemma}
\begin{proof}
For $k=3$, we have equality.
For $k\ge 4$, we use the estimate $\arccos(x)\le \pi - \sqrt{2x+2}$, and immediately see that $\arccos\left(\frac{-(k-1)}{k+1}\right) \le \pi - \frac{2}{\sqrt{k+1}}$.
As $k\ge 4$, we have $4k^2 \ge \pi^2(k+1)$, which implies that $\frac{2}{\sqrt{k+1}} \ge \frac{\pi}{k}$. 
Hence $\arccos\left(\frac{-(k-1)}{k+1}\right) \le \pi - \frac{\pi}{k} = \frac{(k-1)\pi}{k}$ as desired.
\end{proof}

\begin{theorem}\label{thm:1tok-easy}
For $k\ge 2$, we have $2\cdott d_\gh(S^1, S^k) \le \frac{k\pi}{k+1}$.
\end{theorem}
\begin{proof}
Let $P\subseteq S^1$ consist of $2(k+1)$ evenly spaced points, and let $Q\subseteq S^k$ be the standard basis vectors and their negatives.
The correspondence $R_{P,Q}$ of \Cref{thm:induced-correspondence} (shown in \Cref{fig:1to2}) will have distortion at most the maximum of $\vdiam(P)$, $\pi-\sep(P)$, $\vdiam(Q)$, and $\pi-\sep(Q)$.
We may immediately note that $\vdiam(P) = \sep(P) = \frac{\pi}{k+1}$, and $\sep(Q) = \frac{\pi}{2}$.
\Cref{lem:crosspolytope} tells us that $\vdiam(Q) = \arccos\left(\frac{-(k-1)}{k+1}\right)$.
When $k=2$, this is $\arccos\left(-\frac{1}{3}\right) < \arccos\left(-\frac{1}{2}\right) = \frac{2\pi}{3} = \frac{k\pi}{(k+1)}$, and for $k\ge 3$ \Cref{lem:calc} provides a slightly stronger upper bound of $\frac{(k-1)\pi}{k}$.
Plugging all of this in, we see that $R_{P,Q}$ has distortion at most $\pi - \sep(P) = \frac{k\pi}{k+1}$ as desired.
\end{proof}

\begin{remark}
The $k=2$ case of \Cref{thm:1tok-easy} shows that $2\cdott d_{\gh}(S^1,S^2) = \frac{2\pi}{3}$. 
Lim, M\'emoli, and Smith~\cite{LMS22} previously constructed correspondences between $S^1$ and $S^2$ with this optimal distortion, and their constructions likewise involved decomposing $S^2$ into chunks that were then collapsed to points in $S^1$. One of their constructions uses six isometric triangular regions, three in the upper hemisphere and three in the lower hemisphere, and another uses images of rectangular planar regions under a certain parametrization of $S^2$ (the latter appears in the arxiv version of~\cite{LMS22}). 

Our correspondence $R_{P,Q}$ uses a new decomposition of $S^2$, consisting of the six centrally projected facets of the cube.
\Cref{fig:1to2} illustrates these decompositions of $S^1$ and $S^2$ that are used to build $R_{P,Q}$. 
\end{remark}

\begin{figure}[h]
    \[
    \includegraphics{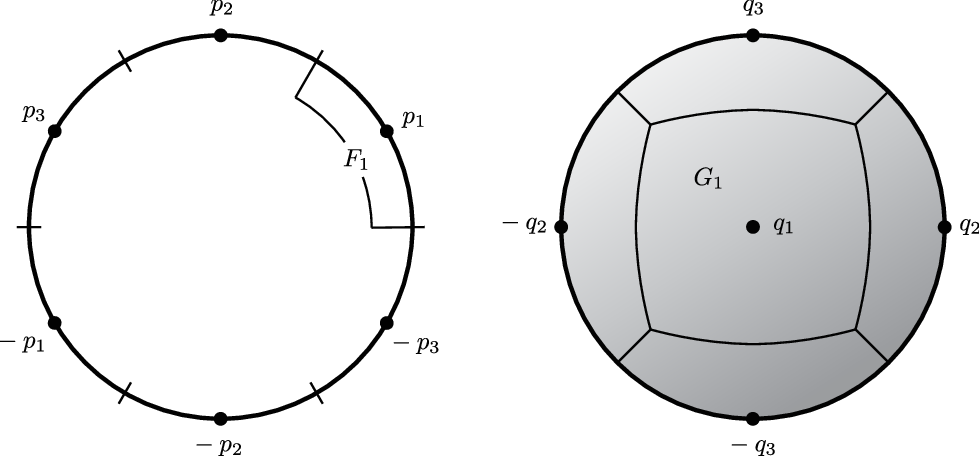}
    \]
    \caption{
    Voronoi cells in $S^1$ and $S^2$ induced by finite antipodal sets $P$ and $Q$ with six points each.
    The correspondence $R_{P,Q}$ of \Cref{thm:induced-correspondence} is optimal in this case, with distortion $\frac{2\pi}{3}$.}
    \label{fig:1to2}
\end{figure}

\begin{remark}\label{rem:kodd}
Note that \Cref{thm:1tok-easy} already implies that $d_{\gh}(S^1, S^{2\ell}) = \frac{2\ell\pi}{2\ell+1}$, which is half of \Cref{thm:1tok}. 
In the following section, we will establish the second half of \Cref{thm:1tok} by showing that   $d_{\gh}(S^1, S^{2\ell+1}) = \frac{2\ell\pi}{2\ell+1}$.
It turns out that \Cref{thm:induced-correspondence} is fundamentally incapable of establishing this result.
Even more strikingly, \Cref{thm:induced-correspondence} is incapable of improving the upper bound  $d_{\gh}(S^1, S^{2\ell+1}) = \frac{(2\ell+1)\pi}{2\ell+2}$ by any small $\varepsilon >0$.

An improvement would require us to construct finite antipodal sets $P\subseteq S^1$ and $Q\subseteq S^{2\ell+1}$ with Voronoi diameter strictly less than $\frac{(2\ell+1)\pi}{2\ell+2}$, and with separation strictly larger than $\frac{\pi}{2\ell+2}$.
By the requirement on separation, such a $P$ can consist of at most $2\ell$ pairs of points. 
Then $Q$ would likewise consist of at most $2\ell$ pairs of points, which would all be contained in an equator of codimension one. 
The poles above this equator would lie in every Voronoi cell associated to points in $Q$, and hence all Voronoi cells would have diameter $\pi$, making \Cref{thm:induced-correspondence} inapplicable.
\end{remark}

\section{From the $1$-sphere to the $k$-sphere: a tight bound for odd $k$}\label{sec:1tok-kodd}

Next we turn our attention to the problem of determining the Gromov--Hausdorff distance between $S^1$ and $S^k$ for odd $k$.
We have seen in \Cref{rem:kodd} that \Cref{thm:induced-correspondence} is incapable of proving \Cref{thm:1tok} for odd $k$.
The goal of this section is to develop a new correspondence $\Rk$ which addresses these issues.
We begin in Section \ref{sec:geom} with a more detailed analysis of the limitations of Theorem \ref{thm:induced-correspondence}; in particular, we see how these limitations suggest a natural geometric construction for the new correspondence $\Rk$.
In Section \ref{sec:geom} we describe $\Rk$ both analytically and geometrically, and in Sections \ref{sec:reduction} and \ref{sec:oddproof} we show that $\Rk$ has the desired distortion.
Section \ref{sec:mainproof} contains the resulting proof of \Cref{thm:1tok}.

\subsection{An improved correspondence}
\label{sec:geom}
In our proof of \Cref{thm:1tok-easy} we considered a correspondence $R_{P,Q}$ where $P\subseteq S^1$ consisted of $2k+2$ evenly spaced points, and $Q\subseteq S^k$ consisted of the vertices of a cross-polytope (i.e. the standard basis vectors and their negatives). 
This correspondence has distortion $\frac{k\pi}{k+1}$, and when $k$ is odd we would like to improve this to $\frac{(k-1)\pi}{k}$. 
We will see that a new geometric idea emerges from studying the precise limitations of the correspondence $R_{P,Q}$.

Consider any pair of points $p_i,p_j \in P$ which lie at distance $\tfrac{k\pi}{k+1}$ on $S^1$.
These points correspond to two nonantipodal Voronoi cells $G_i$ and $G_j$ in $S^k$, which intersect nontrivially (in a codimension $1$ subset), so that, by the definition of $R_{P,Q}$, each $x \in G_i \cap G_j$ corresponds to both $p_i$ and $p_j$.
Thus the distortion of $R_{P,Q}$ is bounded below by $d_{S^1}(p_i,p_j) - d_{S^k}(x,x) = \tfrac{k\pi}{k+1}$, larger than our desired distortion.
Similarly, if we instead choose two points $p_i', p_j'\in P$ at distance $\tfrac{\pi}{k+1}$ in $S^1$, then any pair of antipodal points $x' \in G_i$, $-x' \in G_j$ produces the same lower bound on distortion.

The evident issue is that, by associating all points of each cell $G_i$ to a single point, we have collapsed faraway points $p_i$ and $p_j$ to a single point $x$, and we have spread nearby points $p_i'$ and $p_j'$ to antipodal points $\pm x'$.
This suggests that we should not associate all points of $G_i$ to a single point $p_i \in S^1$, but rather spread the points from $G_i$ across the cell $F_i\subseteq S^1$.

We will achieve this using an embedding $\gamma \colon S^1 \to S^k$, which takes each Voronoi cell $F_i$ in $S^1$ homeomorphically to a diameter of a Voronoi cell $G_i$ in $S^k$.
The embedding of $F_i$ into $G_i$ induces a natural correspondence, defined by sending $x \in G_i$ to $y \in F_i$, where $\gamma(y)$ is the point on $\gamma(F_i)$ which is nearest to $x$.
In the following sections we describe the details of the construction, showing that for an appropriately-chosen embedding $\gamma$, projecting Voronoi cells in $S^k$ onto segments of $\gamma$ yields a correspondence $\Rk$ with the desired distortion.

\paragraph{Notation and definition of $\Rk$.}

We first fix a finite antipodal set $P = \{\pm p_1,\ldots, \pm p_{k+1}\}\subseteq S^1$ where the $p_i$ appear counterclockwise and evenly spaced with alternating signs, starting at $p_1 = e_1$ (see \Cref{fig:notation}).
In $S^k$, we fix $Q = \{\pm q_1, \ldots, \pm q_{k+1}\}\subseteq S^k$, where $q_i = e_i$.
As usual, the Voronoi cells in $S^1$ associated to the $\pm p_i$ are denoted by $\pm F_i$, and similarly the cells in $S^k$ associated to $\pm q_i$ are denoted by $\pm G_i$.
We will frequently use the fact that $G_i$ consists of the points where the $i$-th coordinate is positive and largest in magnitude among all coordinates.

The correspondence $\Rk$ is constructed with respect to a certain embedding $\gamma \colon S^1 \to S^k$, whose image consists of $2k+2$ geodesic arcs in $S^k$, such that each arc is the diameter of a Voronoi cell.
The curve $\gamma$ crosses the  cells in a specific way, and as such, it is convenient to adopt a new notation for the cells $F_i$ and $G_i$, whose linear ordering reflects the traversed order of the cells.
We order the $\pm F_i$ as $\mathcal F_1, \mathcal F_2, \ldots, \mathcal F_{2k+2}$, according to the rule $\mathcal F_m = (-1)^{m+1} F_m$, where the subscript on $F_m$ is understood modulo $k+1$.
Note that this linear order depends on $k$ being odd---for example $\mathcal F_{k+2} (-1)^{k+2}F_1 = -F_1$, whereas if $k$ were even we would have repeated $F_1$ in the linear order and missed $-F_1$.

Explicitly, the linear order on the $\pm F_i$ is as shown below:\begin{align*}
    F_1, -F_2, F_3, -F_4, \ldots, -F_{k+1}, -F_1, F_2, \ldots , F_{k+1}.
\end{align*}
Similarly, we linearly order the $\pm G_i$ as $\mathcal G_1,\mathcal G_2,\ldots, \mathcal G_{2k+2}$, where $\mathcal G_m = (-1)^{m+1} G_m$, and again the latter subscript is understood modulo $k+1$.
We can describe each interval $\mathcal F_m$ explicitly:\[
\Q_m = \bigg[\frac{(2m-3)\pi}{2k+2}, \ \frac{(2m-1)\pi}{2k+2}\bigg].
\]
\Cref{fig:notation} shows the points and cells in $S^1$ when $k=3$.

\begin{figure}[h!t]
    \centering
    \includegraphics{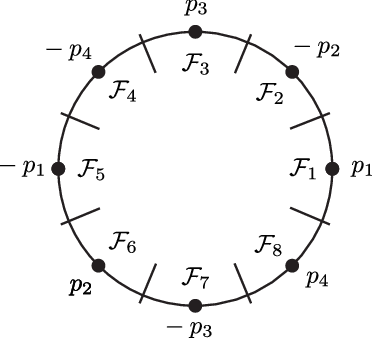}
    \caption{The points $\pm p_1,\pm p_2,\pm p_3, \pm p_4$ in $S^1$, and the linear order on their Voronoi cells.}
    \label{fig:notation}
\end{figure}

Postponing the explicit description of $\gamma$, we instead directly define the functions that result from projecting the various $\mathcal G_m$ onto $\gamma$.
For any $1\le m \le 2k+2$, we define $f_m \colon \Pc_m \subseteq \R^{k+1} \to S^1$ by
\begin{align*}
f_m(x_1,\dots,x_{k+1}) = \frac{(m-1)\pi}{k+1} \ + \ \frac{\pi}{2k(k+1)} \cdot \frac{x_1+\cdots+x_{m-1}-x_{m+1}-\cdots-x_{k+1}}{x_m},
\end{align*}
where the subscripts of $x$ in the latter quotient are considered modulo $k+1$.
Note that since $|x_m| \geq x_i$ for all $i$, this quotient of $x_i$ terms maps onto the interval $[-k,k]$, and so $f_m$ describes a correspondence between $\Pc_m$ and $\Q_m$, as desired.  We also note that the collection of functions $f_m$ exhibits certain natural symmetries, for example, $f_m(x)=-f_{m+k+1}(-x)$ (recall here that $\mathcal G_m = -\mathcal G_{m+k+1}$).
In Section \ref{sec:reduction} we also discuss a certain natural cyclic action compatible with the functions $f_m$ and their domains $\mathcal G_m$.

\begin{definition}\label{def:Rk}
With the notation given above, the correspondence $\Rk$ between $S^k$ and $S^1$ is
\[
\Rk \eqdef \left\{(x,f_m(x)) \in S^k \times S^1 \,\middle\vert\,  x \in \Pc_m\right\}.
\]
\end{definition}

\noindent We note that the correspondence $\Rk$ is not a function: if $x$ lies on the boundary of two or more Voronoi cells, then $x$ is in the domain of multiple $f_m$, and hence corresponds to multiple elements of $S^1$.
In fact, at the beginning of this section, we saw that these boundary points caused undesirable distortion bounds for the correspondence $R_{P,Q}$.
The next example provides evidence that the correspondence $\Rk$ avoids this issue.

\begin{example}
\label{ex:maxdist}
    Let $x \in \Pc_1 \cap \Pc_{k+1}$.  Then $x_1 = -x_{k+1}$, and we compute
    \begin{align*}
    f_{k+1}(x) - f_1(x) & = \frac{k\pi}{k+1} + \frac{\pi}{2k(k+1)} \bigg( \frac{x_1+\cdots+x_k}{x_{k+1}} - \frac{-x_2-\cdots-x_{k+1}}{x_1}\bigg) 
    \\
    & = \frac{k\pi}{k+1} + \frac{\pi}{2k(k+1)} \bigg( -\frac{x_1+\cdots+x_k}{x_1} - \frac{-x_2-\cdots-x_k+x_1}{x_1}\bigg)
    = \frac{(k-1)\pi}{k},
    \end{align*}
the desired distortion bound for $\Rk$.
\end{example}

Our goal in the remainder of this section is to prove that this quantity is exact.

\begin{theorem}
\label{thm:Rdist}
    The distortion of $\Rk$ is equal to $\tfrac{(k-1)\pi}{k}$.
\end{theorem}

\paragraph{Geometric description of $\Rk$.}
Before proceeding to the proof of \Cref{thm:Rdist}, we take a short detour to offer a geometric interpretation of the correspondence $\Rk$, in particular, to see how $\Rk$ arises by projecting the Voronoi cells $\mathcal G_m$ onto a particular embedding $\gamma \colon S^1 \to S^k$.

\begin{remark}
Facundo M\'emoli and Zane Smith~\cite{memolismith24} have explored similar ideas in the form of ``embedding projection correspondences'', which arise by projecting points in a metric space to the nearest point on an embedded copy of another metric space. 
This strategy seems to have broad utility, and has yielded promising experimental results including for correspondences between $S^1$ and $S^k$ arising from piecewise geodesic embeddings $S^1\to S^k$ (``cartoonizations''). 
We first learned of these correspondences during our work on the Gromov--Hausdorff, Borsuk--Ulam, Vietoris--Rips polymath project~\cite{dghpolymath}.
A source of inspiration for our correspondence $\Rk$ were the embedding-projection correspondences of M\'emoli and Smith~\cite{memolismith24}. In fact, $\Rk$ is actually a variant of an embedding-projection correspondence, with the caveat that our nearest-point projection is carefully restricted to Voronoi cells.

\end{remark}

The curve $\gamma$ consists of $2k+2$ geodesic arcs, each of which crosses a diameter of a unique Voronoi cell $\mathcal G_i$.
The corners of the Voronoi cells occur at the $2^{k+1}$ points of $S^k$ where all coordinates have magnitude $\frac{1}{\sqrt{k+1}}$.
We use vectors of $k+1$ signs to denote these vertices---for example, $++-+$ denotes the vertex $\left(\tfrac{1}{2}, \tfrac{1}{2}, -\tfrac{1}{2}, \tfrac{1}{2}\right)$ in $S^3$.
Now $\gamma$ starts from $++\cdots +$ at time $-\frac{\pi}{2k+2}$ (the left endpoint of $\mathcal F_1$), and it crosses Voronoi cells $\mathcal G_i$ in the order indicated by the index, progressing continuously across diameters. 
Note that two diametrically opposite corners of cell $\pm G_i$ have opposite signs in all but the $i$-th coordinate, and so the initial trajectory of $\gamma$ completely determines the $2k+2$ vertices which are visited by $\gamma$.
For example, when $k=3$, the curve $\gamma\colon S^1\to S^3$ visits corners in the following order:
\begin{align*}
    ++++&\ \ \to\ \ \ +---\ \ \to\ \ \ --++\ \ \to\ \ \ +++-\ \ \to\ \ \ ----\\
    &\ \ \to \ \ \ -+++\ \ \to\ \ \ ++--\ \ \to\  \ \ ---+\ \ \to\  \ \ ++++.
\end{align*}
We emphasize here that this construction depends on $k$ being odd: if $k$ were even, then starting from $++\cdots +$ and switching all but one sign $k+1$ times, we would return prematurely to $++\cdots +$ after time $\pi$.  Since $k$ is odd, $\gamma$ instead arrives at $--\cdots -$ after time $\pi$ and continues in a $\Z/2$-equivariant manner: $\gamma(-x) = -\gamma(x)$.

An important caveat is that the curve $\gamma$ does not move at constant speed.
Instead, we define $\gamma$ so that it moves at constant speed in $\R^{k+1}$ \emph{after} $\pm G_i$ is centrally projected outward to a facet of the cube $[-1,1]^{k+1}$, via the map $x \mapsto \frac{1}{|x_i|}x$.
With this definition of $\gamma$, we obtain an alternative definition for $\Rk$, defined by associating every point $x \in \mathcal G_m$  to the nearest point on the segment $\gamma(\mathcal F_m)$. 
Indeed, while this geometric description is not necessary for the distortion computation, it can be readily checked that $f_m \colon \mathcal G_m \to S^1$ is obtained by centrally projecting  $\mathcal G_m$ to a facet of the cube $[-1,1]^k$, and then linearly projecting this facet to the diameter $\gamma(\mathcal F_m)$ (see Figure \ref{fig:cube}).
In this way, each point $y\in \mathcal F_m$ is associated to a fiber of the nearest-point map from $\mathcal G_m$ to $\gamma(\mathcal F_m)$---this fiber is a portion of a great $(k-1)$-sphere in $S^k$, orthogonal to $\gamma(\mathcal F_m)$ at the point $\gamma(x)$.

\begin{figure}[h!t]
\centering
\includegraphics{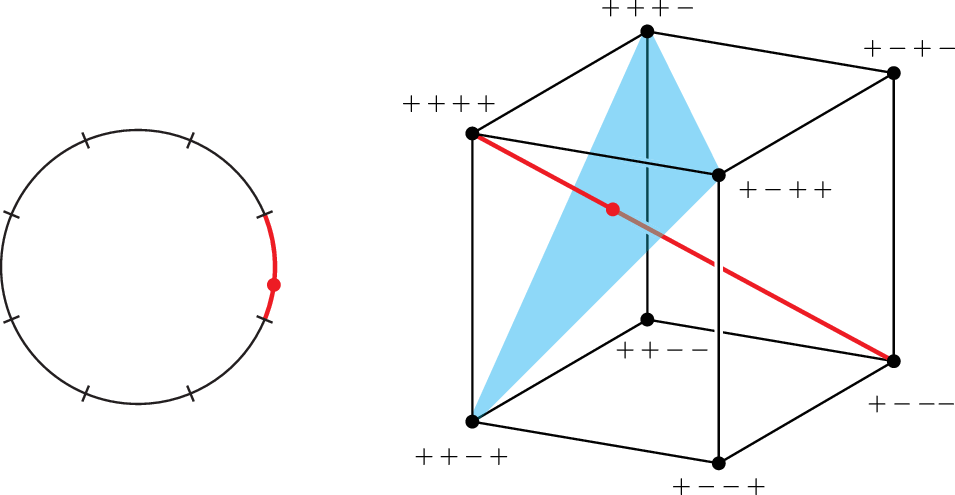}
\caption{(Left) The interval $\mathcal F_1 = \left[-\frac{\pi}{8},\frac{\pi}{8}\right]$ in $S^1$, containing the point $-\frac{\pi}{24}$. 
(Right) The Voronoi cell $\mathcal G_1 \subseteq S^3$, consisting of points whose first coordinate is positive and maximum in magnitude, after central projection to the hyperplane $x_1 = 1$ in $\R^4$.
The bold red diameter is the central projection of $\gamma(F_1)$, and the transparent blue triangle is the set of points that are closest to the point $\gamma\left(-\frac{\pi}{24}\right)$ which lies one third of the distance from $++++$ to $+---$.
The blue points in $S^3$ all correspond to $-\frac{\pi}{24}\in S^1$.}
\label{fig:cube}
\end{figure}

\begin{example}
\label{ex:comps}
In this example we again specialize to the case $k = 3$ and we discuss how the value $\tfrac{2\pi}{3}$ (the claimed distortion) arises naturally here.
Eight of the sixteen vertices in $S^3$ lie on the curve $\gamma$.
Consider $x = (\frac12, -\frac12,\frac12,\frac12)$, which does not lie on the curve $\gamma$.
By checking the signs of the components, we see that $x$ is a corner of the cells $G_1 = \Pc_1$, $-G_2 = \Pc_2$, $G_3 = \Pc_3$, and $G_4 = \Pc_8$, and $-x$ is a corner of $\Pc_4$, $\Pc_5$, $\Pc_6$, and $\Pc_7$.  
By the definition of $\Rk$, we can compute the four correspondents of $x$ using the various functions $f_m$.
For example, one correspondent is \[
f_1(x) = \frac{\pi}{24}\cdot \frac{\tfrac{1}{2}-\tfrac{1}{2}-\tfrac{1}{2}}{\tfrac{1}{2}} = -\frac{\pi}{24}
\]
Computing the remaining correspondents, we find that $x$ corresponds to the four points $-\tfrac{\pi}{24}$,$\tfrac{7\pi}{24}$,$\tfrac{11\pi}{24}$,$\tfrac{43\pi}{24}$, and $-x$ corresponds to the four points $\tfrac{19\pi}{24}$,$\tfrac{23\pi}{24}$,$\tfrac{31\pi}{24}$,$\tfrac{35\pi}{24}$.
The correspondents of $x$ are pairwise closer than $\tfrac{2\pi}{3}$, and the correspondents of $-x$ are pairwise closer than $\tfrac{2\pi}{3}$, whereas each correspondent of $x$ lies at least $\tfrac{\pi}{3}$ away from each correspondent of $-x$.

For a qualitatively distinct example, we consider a Voronoi vertex which does lie on the curve $\gamma$, such as $x' = (\frac12,\frac12,\frac12,\frac12)$.  We find that the three correspondents of $x$, together with the three correspondents of $-x$, form a set of six equally spaced points on $S^1$.
\Cref{fig:comps} shows these two situations.
\end{example}

\begin{figure}[h!t]
\centering
\includegraphics{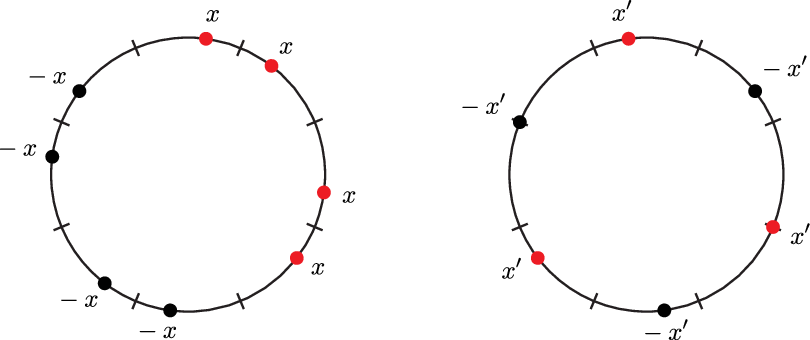}
\caption{Correspondents of the points $\pm x$ (left) and $\pm x'$ (right) described in \Cref{ex:comps}.}
\label{fig:comps}
\end{figure}

\subsection{The distortion of $\Rk$: initial steps}
\label{sec:reduction}

In the following sections, we will study the distortion of $\Rk$ by explicitly analyzing the behavior of the functions $f_m$.
Much of this analysis can be simplified by appealing to certain natural symmetries exhibited by these functions.
After discussing these geometric properties of $f_m$, we demonstrate how these symmetries lead to a significant case reduction in the proof of Theorem \ref{thm:Rdist}.

\paragraph{Symmetries and geometric properties of the $f_m$.}
The functions $f_m$ and their domains $\Pc_m$ are compatible with a certain natural cyclic action, which generalizes the aforementioned $\Z/2$-equivariance $f_m(x) = -f_{m+k+1}(-x)$.
Define the isometry
\[
A_1 \colon \R^{k+1} \to \R^{k+1} \colon (x_1,\dots,x_{k+1}) \mapsto (x_{k+1},-x_1,\dots,-x_k),
\]
and note that $A_1$ takes $\Pc_m$ to $\Pc_{m+1}$, where the subscript on $\mathcal G_{m+1}$ is  taken modulo $2k+2$.
Let $A_n$ refer to the $n$-fold composition $A_1 \circ \cdots \circ A_1$; for example, $A_{k+1}$ is the antipodal map, and $A_{2k+2}$ is the identity.
We have the following lemma, immediate from the definitions of $\Pc_m$ and $f_m$, which we record for future use.
\begin{lemma}
    \label{lem:cyclic}
    For each $n$, $A_n$ takes $\Pc_m$ to $\Pc_{m+n}$, and for $x \in \Pc_m$, $f_m(x) = \frac{n\pi}{k+1} + f_{m+n}(A_n (x))$.
\end{lemma}

Next, we show that each function $f_m$ is strictly distance-decreasing.

\begin{lemma}
\label{lem:distdec}
    Each $f_m$ is distance-decreasing.  That is, for $x \neq y \in \Pc_m$, $d_{S^1}(f_m(x),f_m(y)) < d_{S^3}(x,y)$.
\end{lemma}

\begin{proof}
By Lemma \ref{lem:cyclic}, and since each $A_n$ is an isometry, it is enough to argue the case $m=1$.
We compute for $x \neq y$:
\begin{align*}
|f_1(x) - f_1(y)| & = \frac{\pi}{2k(k+1)} \bigg| \frac{x_2 + \cdots + x_{k+1}}{x_1} - \frac{y_2 + \cdots + y_{k+1}}{y_1} \bigg|
\\
& =  \frac{\pi}{2k(k+1)} \bigg| \frac{y_1(x_2+\cdots+x_{k+1}) - x_1(y_2+\cdots+y_{k+1})}{x_1y_1} \bigg|
\\
& <  \frac{1}{\sqrt{k}(k+1)} \bigg|\frac{y_1(x_2+\cdots+x_{k+1}) - x_1(y_2+\cdots+y_{k+1})}{x_1y_1}\bigg|,  \hspace{.1in} \mbox{since } k \geq 3 \Rightarrow \pi < 2\sqrt{k}
\\
& \leq \frac{1}{\sqrt{k}} | y_1(x_2+\cdots+x_{k+1}) - x_1(y_2+\cdots+y_{k+1})|, \hspace{.6in} \mbox{since } x_1y_1 \geq \tfrac{1}{k+1}
\\
& = \frac{1}{\sqrt{k}} | x_2y_1 - x_1y_2 + \cdots + x_{k+1}y_1 - x_1y_{k+1}|
\\
& \leq \frac{1}{\sqrt{k}} \Big(|x_2y_1 - x_1y_2| + \cdots + |x_{k+1}y_1 - x_1y_{k+1}|\Big).
\end{align*}
Let $u$ and $v$ be the $k$-tuples $u = (|x_2y_1-x_1y_2|, \dots, |x_{k+1}y_1-x_1y_{k+1}|)$ and $v=(1,\dots,1)$.
A standard application of Cauchy--Schwarz recovers a well-known comparison between $1$- and $2$-norms:
\begin{align}
\label{eqn:normeq}
\|u\|_1 = \langle u, v \rangle \leq \|u\|_2\|v\|_2 = \sqrt{k}\|u\|_2.
\end{align}
Applying this to the final term of our inequality, we continue:
\begin{align*}
& \frac{1}{\sqrt{k}} \Big(|x_2y_1 - x_1y_2| + \cdots + |x_{k+1}y_1 - x_1y_{k+1}|\Big)
\\
\leq \ & \sqrt{\sum_{2 \leq i \leq k+1} (x_iy_1-x_1y_i)^2}
\\
\leq \ & \sqrt{\sum_{1 \leq i < j \leq k+1} (x_iy_j-x_jy_i)^2}
\\
= \ & \sqrt{1-\langle x, y \rangle^2} \mbox{\hspace{2in} by Lagrange's identity}
\\
= \ & \sin(\arccos(\langle x, y \rangle)
\\
< \ & \arccos(\langle x, y \rangle),
\end{align*}
as desired.
\end{proof}

\begin{remark} The proof of Lemma \ref{lem:distdec} is fairly representative of the techniques used to bound the distortion of $\Rk$, and we will continue to see geometric terms, such as the area $\sqrt{1-\langle x , y \rangle^2}$ of the parallelogram spanned by $x$ and $y$, arise in these comparisons.
A powerful consequence of Lemma \ref{lem:distdec} is that the distortion of $\Rk$ can only be achieved at points of $S^k$ lying on the boundary of two or more cells (see Proposition \ref{prop:maxdist}).
We believe that this observation, combined with a careful analysis of the boundary components, could lead to a full proof of Theorem \ref{thm:Rdist}.
Our current proof relies on this observation only in the case $k=3$, and for $k \geq 5$ we are able to give a more streamlined approach.
\end{remark}

\paragraph{A case reduction.} Here we demonstrate how the symmetries discussed above can be used to simplify the distortion analysis for $\Rk$.
We have shown in Example \ref{ex:maxdist} that the distortion of $\Rk$ is at least $\tfrac{(k-1)\pi}{k}$, and it remains to show that this quantity is an upper bound.
To achieve this, we must show that, for each pair $(i,j)$ with $1\le i \le j \le 2k+2$, the ``distortion function" 
\begin{align*}
D_{i,j} \colon \Pc_i \times \Pc_j \to [0,\pi], \hspace{.3in} D_{i,j}(x,z) \eqdef \big| d_{S^1}\big(f_i(x),f_j(z)\big) - d_{S^k}(x,z)\big|
\end{align*}
is bounded above by $\frac{(k-1)\pi}{k}$. 
That is, we must show that
\begin{align}
\label{eqn:allij}\tag{$\ast$}
    D_{i,j}(x,z) \leq \frac{(k-1)\pi}{k} \hspace{.25in} \mbox{ for all } x\in\Pc_i, z\in\Pc_j.
\end{align}

Using symmetries and some other simple observations, we are able to significantly reduce the number of pairs $(i,j)$ which must be checked.

\begin{proposition}
\label{prop:casered}
For the collection of functions $D_{i,j}$, the following statements hold:
\begin{enumerate}[(a)]
\item $D_{i,j}$ satisfies \emph{(\ref{eqn:allij})} \ $\Longleftrightarrow$ \ $D_{1,j-i+1}$ satisfies \emph{(\ref{eqn:allij})} \ $\Longleftrightarrow$ \ $D_{1,i-j+1}$ satisfies \emph{(\ref{eqn:allij})}.
\item $D_{1,j}$ satisfies \emph{(\ref{eqn:allij})} \ $\Longleftrightarrow$ \ $D_{j,k+2}$ satisfies \emph{(\ref{eqn:allij})} \ $\Longleftrightarrow$ \ $D_{1,k+3-j}$ satisfies \emph{(\ref{eqn:allij})}.
\item For $j \in \left\{1, 4, 5, \dots, k-2,k-1,k+2\right\}$, $D_{1,j}$ satisfies \emph{(\ref{eqn:allij})}.
\end{enumerate}

\end{proposition}

\begin{proof}
The first two items are consequences of certain natural symmetries, which can be used to show that the claimed functions have identical images; the final item comes from a quick observation about the image of $D_{1,j}$ for those $j$ in the hypothesis of the statement. \medskip

\noindent For part (a), we apply Lemma \ref{lem:cyclic} and the fact that each $A_n$ is an isometry.
Indeed, we have
\begin{align*}
D_{1,j-i+1}(x,z) & = \big|d_{S^1}\big(f_1(x),f_{j-i+1}(z)\big) - d_{S^k}(x,z)\big|
\\
& = \big|d_{S^1}\big(f_i(A_{i-1}x),f_j(A_{i-1}z)\big) - d_{S^k}\big(A_{i-1}x,A_{i-1}z\big)\big|
\\
& = D_{i,j}\big(A_{i-1}x,A_{i-1}z\big),
\end{align*}
and a similar argument verifies the statement for $D_{1,i-j+1}$. \medskip

\noindent For part (b), we take advantage of the aforementioned $\Z/2$-equivariance in the collection $f_m$: for $x \in \Pc_1$, $f_1(x) = -f_{k+2}(-x)$ (here, recall that the cell $\Pc_{k+2}$ is the antipode of the cell $\Pc_1$).
Therefore
\begin{align*}
D_{1,j}(x,z) & = \big|d_{S^1}\big(f_1(x),f_{j}(z)\big) - d_{S^k}(x,z)\big|
\\
& = \big| \big( \pi - d_{S^1}(f_{k+2}(-x),f_j(z)) \big) - \big( \pi - d_{S^k}(-x,z) \big) \big|
\\
& = D_{j,k+2}(-x,z),
\end{align*}
establishing the first claimed equivalence.
The second claimed equivalence follows from part (a).  \medskip

\noindent For part (c), when $j = 1$ or $j=k+2$: if $x,z \in \Pc_1$, then
\[
d_{S^1}(f_1(x),f_1(z)) \leq \tfrac{\pi}{k+1} \leq \tfrac{(k-1)\pi}{k}
\]
and
\[
d_{S^k}(x,z) \leq \diam(\Pc_1) = \arccos(\tfrac{1-k}{k+1}) \leq \tfrac{(k-1)\pi}{k}
\]
by Lemma \ref{lem:calc}.
Therefore $D_{1,1}$ satisfies (\ref{eqn:allij}), and by part (b), so does $D_{1,k+2}$.  \medskip

\noindent For part (c), when $4\le j \le k-1$: observe that $d_{S^1}(f_1(x),f_j(z))$ takes values in the interval $\Big[\frac{(j-2) \pi}{k+1}, \frac{j \pi}{k+1}\Big]$, while $d_{S^k}(x,z)$ takes values in the interval $[0,\pi]$.
Therefore the image of $D_{1,j}$ is bounded above by
\[
\max\bigg\{ \frac{j\pi}{k+1}, \frac{(k-j+3)\pi}{k+1} \bigg\}.
\]
When $j \in \left\{4,5,\dots,k-2,k-1\right\}$, each of these terms is bounded above by $\frac{(k-1)\pi}{k+1} \leq \frac{(k-1)\pi}{k}$.
\end{proof}

\subsection{The proof of Theorem \ref{thm:Rdist}}
\label{sec:oddproof}

With the case reduction of the previous section, we will see that we need only prove Theorem \ref{thm:Rdist} analytically in two cases.
We record a few simple bounds to ease the main proof.
Recall that for $z \in \Pc_m$, $|z_m| \geq |z_i|$ for all $i$, and in particular, $|z_m| \geq \tfrac{1}{\sqrt{k+1}}$, where, as usual, we understand subscripts on coordinates modulo $k+1$.

\begin{lemma}
    \label{lem:eucldistbound}
    Let $x \in \Pc_1$ and $z \in \Pc_m$.  
    \begin{enumerate}[(a)]
    \item $(x_1 - |z_m|)^2 \leq (x_1 - z_1)^2 + (x_m - z_m)^2 \leq \langle x - z, x - z \rangle$,
    \item For $k \geq 5$, $1+x_1|z_m| \leq \tfrac{4k^2(k+1)}{\pi^2}x_1^2z_m^2$,
    \item For $k \geq 3$, $\tfrac12+x_1|z_m| \leq \tfrac{4k^2(k+1)}{\pi^2}x_1^2z_m^2$
    \end{enumerate}
\end{lemma}

\begin{proof} For part (a), we consider two cases:
\begin{itemize}
    \item if $x_1 \geq |z_m|$ then $(x_1 - |z_m|)^2 \leq (x_1 - z_1)^2$ since $|z_m| \geq z_1$;
    \item if $|z_m| \geq x_1$, then $(x_1 - |z_m|)^2 \leq (z_m - x_m)^2$ since $x_1 \geq x_m$.
\end{itemize}

\noindent For part (b), since $x_1|z_m| \geq \tfrac{1}{k+1}$, we have
\[
\frac{1+x_1|z_m|}{x_1^2z_m^2} = \frac{1}{x_1^2z_m^2} + \frac{1}{x_1|z_m|} \leq (k+1)^2 + (k+1) = (k+1)(k+2) \leq (k+1)\tfrac{4k^2}{\pi^2},
\]
where the final inequality requires $k \geq 5$. \\

\noindent For part (c), similarly, we have
\[
\frac{1+2x_1|z_m|}{2x_1^2z_m^2} = \frac{1}{2x_1^2z_m^2} + \frac{1}{x_1|z_m|} \leq \tfrac12 (k+1)^2 + (k+1) = (k+1)(\tfrac12 k+ \tfrac32) \leq k(k+1) \leq (k+1)\tfrac{4k^2}{\pi^2},
\]
where the final inequality holds for all $k \geq 3$.
\end{proof}

To prove Theorem \ref{thm:Rdist}, we must show that $D_{i,j}$ satisfies (\ref{eqn:allij}) for all $1\le i\le j \le 2k+2$.
By \Cref{prop:casered}, part (a), it suffices to consider the cases $i=1$ and $1 \leq j \leq k+2$.
By \Cref{prop:casered}, part (c), it suffices to consider $j \in \left\{2, 3, k, k+1\right\}$, and part (b) further reduces this to the two cases $j = k$ and $j = k+1$.
The following lemma gives a sufficient condition for our desired inequality in each of these cases.

\begin{lemma}
\label{lem:simplerform}To show that $D_{1,j}(x,z) \leq \tfrac{(k-1)\pi}{k}$, it suffices to show, \smallskip
\begin{align*}
& \mbox{(a) for } j=k\colon 
&& \frac{\pi}{2k(k+1)}\left(\frac{x_1 +\cdots +x_{k+1}}{x_1} + \frac{z_1+\cdots + z_k - z_{k+1}}{z_k} -2k \right) \le \sqrt{\langle x - z, x - z \rangle}. \\
& \mbox{(b) for } j=k+1\colon \ &&
\frac{\pi}{2k(k+1)}\left(\frac{x_1 +\cdots +x_{k+1}}{x_1} + \frac{z_1+\cdots+z_{k+1}}{z_{k+1}} \right) \leq \sqrt{\langle x-z, x-z \rangle}.
\end{align*}
\end{lemma}
\begin{proof}
For $j=k$ and $j=k+1$, the function $D_{1,j} \colon \Pc_1 \times \Pc_j \to [0,\pi]$ is equal to $|d_{S^1}(f_1(x),f_j(z)) - d_{S^{k}}(x,z)|$.
In both cases, the circle distance is sufficiently large (namely, at least $\tfrac{\pi}{k}$), and it suffices to show that $f_j(z)-f_1(x) - \arccos(\langle x, z\rangle$ is bounded above by $\tfrac{(k-1)\pi}{k}$.
Rearranging, it equivalently suffices to show that 
\[
-\tfrac{(k-1)\pi}{k} + f_j(z) - f_1(x) \leq \arccos(\langle x, z \rangle)
\]
for $j=k$ and $j=k+1$.

For part (a), we compute using the definition of $f_1$ and $f_k$ that $-\tfrac{(k-1)\pi}{k} + f_j(z) - f_1(x)$ is equal to:
\begin{align*}
    & -\frac{(k-1)\pi}{k} + \frac{(k-1)\pi}{k+1} + \frac{\pi}{2k(k+1)}\left(\frac{z_1 + \cdots + z_{k-1} - z_{k+1}}{z_{k}} - \frac{-x_2-x_3-\cdots -x_{k+1}}{x_1}\right)\\
    = \ \ & \frac{(2-2k)\pi}{2k(k+1)}  + \frac{\pi}{2k(k+1)}\left(\frac{z_1 + \cdots + z_{k-1} - z_{k+1}}{z_{k}} + \frac{x_2+x_3+\cdots +x_{k+1}}{x_1}\right)\\
    = \ \ & \frac{\pi}{2k(k+1)}\left(\frac{x_1 +\cdots +x_{k+1}}{x_1} + \frac{z_1+\cdots + z_k - z_{k+1}}{z_k} -2k \right),
\end{align*}
where in the final step we have incorporated part of the constant term into the quotients.
On the other hand, since the Euclidean distance between $x\neq z$ is shorter than the spherical distance, we have  $\sqrt{\langle x-z,x-z\rangle} \le \arccos(\langle x,z\rangle)$, giving the desired result.
A similar computation justifies (b).
\end{proof}

\noindent With the case reduction and setup complete, we are ready to prove \Cref{thm:Rdist}.\medskip

\begin{proof}[Proof of Theorem \ref{thm:Rdist}.] We consider each case from \Cref{lem:simplerform} separately below.
Our argument for the case $j=k$ works for any value of $k$ (Case 1, below). 
For $j=k+1$, we make two separate arguments: one when $k\ge 5$ (Case 2) and one when $k = 3$ (Case 3).\medskip

\noindent \textbf{Case 1:} $j=k$:
Starting with the left side of Lemma \ref{lem:simplerform}(a), we write
\begin{align*}
& \frac{\pi}{2k(k+1)}\left(\frac{x_1 +\cdots +x_{k+1}}{x_1} + \frac{z_1+\cdots + z_k - z_{k+1}}{z_k} -2k \right) 
\\
\leq \ & \frac{\pi}{2k(k+1)} \left(\frac{x_{k+1}}{x_1} - \frac{z_{k+1}}{z_k}\right) &&\text{since \ $\frac{x_i}{x_1}\le 1$ \ and \ $\frac{z_j}{z_k}\le 1$}
\\
= \ & \frac{\pi}{2k(k+1)} \left(\frac{x_{k+1}z_k - x_1z_{k+1}}{x_1z_k} \right)
\\
\leq \ & |x_{k+1}z_k - x_1z_{k+1}| &&\text{since \ $\frac{1}{|x_1z_k|} \leq k+1$.}
\end{align*}
This final quantity is the area of the planar parallelogram spanned by the vectors $(x_1-z_k,x_{k+1}-z_{k+1})$ and $(z_k,z_{k+1})$, and hence is bounded above by the product of side lengths
\begin{align*}
& \sqrt{\big((x_1-z_k)^2 + (x_{k+1}-z_{k+1})^2\big)\big(z_k^2+z_{k+1}^2\big)}
\\
\leq \ & \sqrt{(x_1-z_k)^2 + (x_{k+1}-z_{k+1})^2} && \mbox{since } z_k^2+z_{k+1}^2 \leq 1
\\
\leq \ & \sqrt{(x_1-z_1)^2 + (x_k-z_k)^2 + (x_{k+1}-z_{k+1})^2} &&\mbox{by \Cref{lem:eucldistbound}(a)}
\\
\leq \ & \sqrt{\langle x - z, x - z \rangle}.
\end{align*}
By Lemma \ref{lem:simplerform}(a), this concludes Case 1. \medskip

\noindent \textbf{Case 2: } $j = k+1$ and $k \geq 5$: Starting with the left side of Lemma \ref{lem:simplerform}(b), we write
\begin{align*}
& \frac{\pi}{2k(k+1)}\left(\frac{x_1+\cdots +x_{k+1}}{x_1} + \frac{z_1+\cdots+z_{k+1}}{z_{k+1}} \right)
\\
\leq \ & \frac{\pi}{2k(k+1)}\sum_{i=1}^{k+1}\Big|\frac{x_i}{x_1} + \frac{z_i}{z_{k+1}}\Big|
\\
 \leq \ & \frac{\pi}{2k\sqrt{k+1}}\sqrt{\sum_{i=1}^{k+1}\left(\frac{x_i}{x_1} + \frac{z_i}{z_{k+1}} \right)^2} && \mbox{by comparison of norms; see (\ref{eqn:normeq}) }
 \\
= \ & \frac{\pi}{2k\sqrt{k+1}}\sqrt{\frac{z_{k+1}^2+x_1^2+2x_1z_{k+1}\langle x, z \rangle}{x_1^2z_{k+1}^2}} && \mbox{since } x \mbox{ and } z \mbox{ are unit.}
\end{align*}
Focusing on the numerator of the radicand, we write
\begin{align*}
& z_{k+1}^2+x_1^2+2x_1z_{k+1}\langle x, z \rangle
\\
= \ & (z_{k+1}+x_1)^2-2x_1z_{k+1}+2x_1z_{k+1}\langle x, z \rangle
\\
= \ & (z_{k+1}+x_1)^2-x_1z_{k+1}\langle x-z, x-z \rangle
\\
\leq \ & \langle x-z, x-z \rangle(1-x_1z_{k+1}) && \mbox{by Lemma \ref{lem:eucldistbound}(a)}
\\
\leq \ & \langle x-z, x-z \rangle x_1^2z_{k+1}^2(k+1)\tfrac{4k^2}{\pi^2} && \mbox{by Lemma \ref{lem:eucldistbound}(b), since }  k \geq 5.
\end{align*}
Plugging back into the radicand and simplifying concludes Case 2. \medskip

\noindent \textbf{Case 3: } $j = k+1$ and $k = 3$:
Here we present a slight modification that allows us to push the argument through in the case $k=3$.
This modification relies on the fact that each $f_m$ is distance-decreasing (Lemma \ref{lem:distdec}), from which it follows that the maximum value of each $D_{i,j} \colon \Pc_i \times \Pc_j \to [0,\pi]$ occurs only at boundary points
(see Proposition \ref{prop:maxdist} below).

Assuming the result of Proposition \ref{prop:maxdist}, the final two steps of Case 2 can be modified as follows.
If $x \in \Pc_1 \cap \Pc_i$ for some $i$ and $z \in \Pc_{k+1} \cap \Pc_j$ for some $j$, then $(x_1 - z_1)^2$ and $(x_i-z_i)^2$ are both at least $(x_1+z_{k+1})^2$ by a similar argument to the proof of \Cref{lem:eucldistbound}(a).
Thus $2(x_1+z_{k+1})^2 \leq \langle x - z, x - z \rangle$, and the last three lines of the proof of Case 2 can be replaced by
\begin{align*}
& (z_{k+1} + x_1)^2 - x_1z_{k+1}\langle x-z,x-z\rangle\\
\leq \ & \langle x-z, x-z \rangle(\tfrac12-x_1z_{k+1})\\
\leq \ & \langle x-z, x-z \rangle x_1^2z_{k+1}^2(k+1)\tfrac{4k^2}{\pi^2}
\end{align*}
where the final inequality is due to \Cref{lem:eucldistbound}(c).
\end{proof}

\begin{proposition}
\label{prop:maxdist}
The maximum value of each $D_{i,j} \colon \Pc_i \times \Pc_j \to [0,\pi]$ occurs only at boundary points, i.e. when $x \in \partial \Pc_i$ and $z \in \partial \Pc_j$.
\end{proposition}
\begin{proof}
    By Proposition \ref{prop:casered}(a), it suffices to consider the case $i=1$.
    Suppose that $x$ is an interior point of $\Pc_1$ and $z$ is any point in $\Pc_j$. First observe that if $d_{S^1}(f_1(x),f_j(z)) = d_{S^k}(x,z)$, then $D_{1,j}(x,z) = 0$ is not a maximum value.
    
    Next suppose that $d_{S^1}(f_1(x),f_j(z)) > d_{S^k}(x,z)$.
    This implies that $x$ and $z$ are not antipodal, and we choose $y \in \Pc_1$ lying on the minimal geodesic arc $c$ connecting $x$ to $z$, sufficiently close to $x$ so that $d_{S^1}(f_1(y),f_j(z)) \geq d_{S^k}(y,z)$.
    We compute
    \begin{align*}
    D_{1,j}(y,z) - D_{1,j}(x,z)     & = d_{S^1}(f_1(y),f_j(z)) - d_{S^k}(y,z) - \big(d_{S^1}(f_1(x),f_j(z))  - d_{S^k}(x,z)\big) \\
                        & = d_{S^k}(x,y) + d_{S^1}(f_1(y),f_j(z)) - d_{S^1}(f_1(x),f_j(z)) \\
                        & \geq d_{S^k}(x,y) - d_{S^1}(f_1(x),f_1(y)) && \hspace{-7em} \mbox{ by the triangle inequality,}
                        \\
                        & > 0 && \hspace{-7em}\mbox{ by Lemma \ref{lem:distdec}},
    \end{align*}
    verifying that the value of $D_{1,j}$ increases when $x$ is moved towards $z$.
    
    In fact, we have shown that $d_{S^1}(f_1(y),f_j(z)) > d_{S^k}(y,z)$, which implies that the set of $y$ lying on $c \cap \Pc_1$ and satisfying $d_{S^1}(f_1(y),f_j(z)) \geq d_{S^k}(y,z)$ is nonempty, open, and closed in $c \cap \Pc_1$, hence is equal to $c \cap \Pc_1$.
    Therefore the value of $D_{1,j}$, when restricted to the geodesic arc $c \cap \Pc_1$, is strictly increasing from $x$ to the boundary of $\Pc_1$. 
    
    Now suppose that $d_{S^1}(f_1(x),f_j(z)) < d_{S^k}(x,z)$.
    This implies that $x$ and $z$ are not antipodal, as follows: since $x$ lies in the interior of $\Pc_1$, antipodality would imply that $z = -x$ lies in the interior of $-\Pc_1 = \Pc_{k+2}$, but $D_{1,k+2}(x,-x) = 0$ by Lemma \ref{lem:cyclic}, contradicting our supposition.
    Thus we can choose $y \in \Pc_1$ lying on the minimal geodesic arc connecting $x$ to $-z$; that is, here we move $x$ away from $z$.
    We compute similarly:
    \begin{align*}
    D_{1,j}(y,z) - D_{1,j}(x,z)     & = d_{S^k}(y,z) - d_{S^1}(f_1(y),f_j(z)) - \big(d_{S^k}(x,z) - d_{S^1}(f_1(x),f_j(z))\big) \\
    & = d_{S^k}(x,y) + d_{S^1}(f_1(x),f_j(z)) - d_{S^1}(f_1(y),f_j(z)) \\
                        & \geq d_{S^k}(x,y) - d_{S^1}(f_1(x),f_1(y)) > 0.
    \end{align*}
    By the same argument as above, we conclude that $x$ must be a boundary point of $\mathcal G_1$ to maximize $D_{1,j}(x,z)$.
    A symmetric argument shows that $z$ must be a boundary point of $\Pc_j$.
\end{proof}

\subsection{The proof of Theorem \ref{thm:1tok}}
\label{sec:mainproof}
We are now prepared to complete the proof of Theorem \ref{thm:1tok}.
\begin{proof}[Proof of Theorem \ref{thm:1tok}]
Using the definition of the Gromov--Hausdorff distance (\ref{eq:dgh}), we have
\[
\text{for even } k\colon \hspace{.2in} 2\cdott d_\gh(S^1,S^k) \leq \dis(R_{P,Q}) = \frac{k\pi}{k+1}, \hspace{.5in} \text{by Theorem \ref{thm:1tok-easy}},
\]
and
\[
\text{for odd } k\colon \hspace{.22in} 2\cdott d_\gh(S^1,S^k) \leq \dis(\Rk) = \frac{(k-1)\pi}{k}, \hspace{.4in} \text{by Theorem \ref{thm:Rdist}}.
\]
The fact that these estimates are sharp follows from combining Main Theorem and Theorem 5.1 of~\cite{dghpolymath}.
\end{proof}

\section{General upper bounds: proving \Cref{thm:ntok} and \Cref{thm:packing}}\label{sec:general}

To prove \Cref{thm:ntok}, we will construct point sets in $S^n$ (with $n\ge 2$) with Voronoi diameter at most $\frac{k\pi}{k+1}$, and separation at least $\frac{\pi}{k+1}$. 
These sets will consist of $2(k+1)$ points, matching the number of vertices of the cross-polytope in $S^k$, so that we may apply \Cref{thm:induced-correspondence}. 
It turns out that there is a relatively simple construction: start with the vertices of a cross-polytope in $S^n$, then add the remaining points along the projected edges of the cross-polytope, spacing them as evenly as possible.
The following lemma explains this construction.

\begin{lemma}\label{lem:arcs}
For each $2\le n < k < \infty$, there exists an antipodal set $P\subseteq S^n$ consisting of $2(k+1)$ points, with \[
\vdiam(P)\le \frac{\pi n}{n+1}\quad \text{ and } \quad \sep(P) \ge \frac{\pi}{2\left( \left\lceil \frac{k-n}{n(n+1)}\right\rceil + 1\right)}\ge \frac{\pi}{k-n+3}.\]
\end{lemma}
\begin{proof}
To start, set $P = \{\pm e_1,\ldots, \pm e_{n+1}\}\subseteq S^n$. 
By \Cref{lem:crosspolytope} we have $\vdiam(P)\le \frac{\pi n}{n+1}$, and adding further points can only improve this inequality.
We must add $2(k+1)-2(n+1) = 2(k-n)$ points to $P$. 
Consider the $2n(n+1)$-many geodesic arcs in $S^n$ between pairs of points $\{\pm e_i, \pm e_j\}$, where $i\neq j$. 
Split these arcs into antipodal pairs, choosing one element of each pair as the ``positive" copy of the arc.
Set $N\eqdef \left\lceil \frac{k-n}{n(n+1)}\right\rceil$. 
Into the interior of each of the $n(n+1)$-many positive arcs, place up to $N$-many points, evenly spaced with distances between consecutive points (including the endpoints of the arc) at least $\frac{\pi}{2(N+1)}$.
Copy the points from the positive arcs antipodally to the negative arcs.
Collecting everything, we have up to $2(k+1)$ points in $P$, as shown in \Cref{fig:arcs}.
\begin{figure}[h]
    \[
    \includegraphics{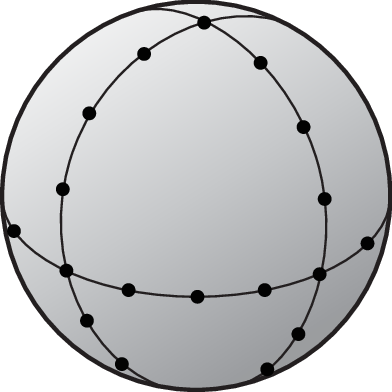}
    \]
    \caption{$N = 3$ points along geodesic arcs between non-antipodal vertices of the octahedron, as in the proof of \Cref{lem:arcs}.}
    \label{fig:arcs}
\end{figure}

Now, we claim that $\sep(P) \ge \frac{\pi}{2(N+1)}$. 
By construction, this bound holds when we restrict our attention to any copy of $S^1$ containing $\{\pm e_i, \pm e_j\}$.
The only other case to worry about is the distance between $p,p'\in P$, where $p$ and $p'$ lie in the interiors of arcs that are not coplanar. 
In particular, the nonzero coordinates in $p$ and $p'$ do not occur in the same indices.
As each point in $P$ has at most two nonzero coordinates, we either have $\langle p, p'\rangle = 0$ or $\langle p, p'\rangle \le c^2$ where $c$ is the maximum absolute value among the coordinates of all points in $P$ that lie on the interiors of arcs.
This implies that $d_{S^n}(p,p')$ is at least as large as the distance from $p$ (respectively, $p'$) to the nearest $\pm e_i$. 
This proves that $\sep(P) \ge \frac{\pi}{2(N+1)}$ as desired. 
The final inequality follows from the fact that $n\ge 2$, so $2\left\lceil \frac{k-n}{n(n+1)}\right\rceil \le k-n+1$. 
\end{proof}

With this construction, we can now recall and prove \Cref{thm:ntok}.

\ntok*
\begin{proof}
The case $n=1$ follows from \Cref{thm:1tok-easy}.
For $n\ge 2$, let $P\subseteq S^n$ be the antipodal set of $2(k+1)$ points of \Cref{lem:arcs}, and let $Q\subseteq S^k$ be the standard basis vectors and their negatives. 
The correspondence $R_{P,Q}$ of \Cref{thm:induced-correspondence} will have distortion at most the maximum of $\vdiam(P)$, $\pi-\sep(P)$, $\vdiam(Q)$, and $\pi-\sep(Q)$.
\Cref{lem:arcs} guarantees that $\vdiam_{S^n}(P)\le \frac{\pi n}{n+1} < \frac{\pi k}{k+1}$, and $\pi-\sep_{S^n}(P)$ is bounded above by $\pi-\frac{\pi}{k-n+3} = \frac{\pi(k-n+2)}{k-n+3}\le \frac{k\pi}{k+1}$. 
\Cref{lem:crosspolytope} guarantees that $\vdiam_{S^k}(Q)$ and $\pi-\sep_{S^k}(Q)$  are bounded above by $\frac{\pi k}{k+1}$. 
The result follows.
\end{proof}

The proof of \Cref{thm:packing} likewise proceeds by choosing appropriate antipodal sets and applying \Cref{thm:induced-correspondence}. 
As before, in the higher dimensional sphere we use the vertices of a cross-polytope.
In the lower dimensional sphere, we simply use an optimal packing of antipodal points. 

\packing*
\begin{proof}
Let $P\subseteq S^n$ be an antipodal set of $2(k+1)$ points whose image $\overline P$ in $\RP^n$ is an optimal packing, with as few points as possible at pairwise distance exactly $p_{k+1}(\RP^n)$. 
Let $Q\subseteq S^k$ be the vertices of a cross-polytope.
Applying \Cref{thm:induced-correspondence}, it will suffice to give appropriate upper bounds on $\vdiam(P)$, $\pi-\sep(P)$, $\vdiam(Q)$, and $\pi-\sep(Q)$. 
The latter two quantities are bounded above by $\arccos\left(\frac{-(k-1)}{k+1}\right)$ by \Cref{lem:crosspolytope}, and so it remains to consider the first two quantities.

The separation of $P$ in $S^n$ is exactly equal to $p_{k+1}(\RP^n)$.
Indeed, the quotient map $S^n \to \RP^n$ exactly preserves the distance between any two points that are within a distance of $\pi/2$ from one another.
As $S^n$ is antipodal, the closest pair of points in $S^n$ has distance at most $\pi/2$ between them.

The Voronoi diameter of $P$ is at most $2d_\HH(P, S^n)$, and we further claim that $d_\HH(P, S^n) \le p_{k+1}(\RP^n)$.
First note that because $P$ is antipodal, $d_\HH(P, S^n)$ is equal to $d_\HH(\overline P, \RP^n)$.
Since $\overline P$ has as few pairs of points as possible at distance exactly $p_{k+1}(\RP^n)$, we see that $d_\HH(\overline P, \RP^n)\le p_{k+1}(\RP^n)$; otherwise we could delete a point in $\overline P$ with distance exactly $p_{k+1}(\RP^n)$ to another point in $\overline P$ and replace it by a point at a strictly larger distance to all other points in $\overline P$.
This proves the result.
\end{proof}

\asymptoticcor*
\begin{proof}
By \Cref{thm:packing}, we see that \begin{align*}
\pi - 2\cdott d_{\gh}(S^n,S^k) &\ge \pi - \max \left \{\arccos\left(\frac{-(k-1)}{k+1}\right) ,\,\, \pi - p_{k+1}(\RP^n),\,\, 2p_{k+1}(\RP^n) \right\}\\
& = \min \left\{\arccos\left(\frac{k-1}{k+1}\right), p_{k+1}(\RP^n), \pi - 2p_{k+1}(\RP^n)\right\}.
\end{align*}
It thus suffices to bound the three quantities above.
Using the bound $\arccos(x) \ge \sqrt{2-2x}$, we see that $\arccos\left(\frac{k-1}{k+1}\right)$ is at least $ \frac{2}{\sqrt{k+1}}$.
The quantity $\pi - 2p_{k+1}(\RP^n)$ tends to $\pi$ as $k\to\infty$, and so can be safely ignored.
Lastly, the fact that $p_{k+1}(\RP^n)$ is bounded below by a multiple of $\frac{1}{\sqrt[n]{k}} \ge \frac{1}{\sqrt{k}}$ completes the proof.
\end{proof}

\begin{remark}
For any $m$, one can chose optimal packings of $m$ points in $\RP^n$ and $\RP^k$, and a similar argument to the one above shows that \[
2\cdott d_\gh(S^n, S^k) \le \max \{\pi-p_m(\RP^n),2p_m(\RP^k)\}. 
\]
Choosing $m = k+1$, we note that $2p_{k+1}(\RP^k) = \pi$, and so this formulation does not generalize or improve on \Cref{thm:packing}.
We are also not aware of examples where choosing some $m > k+1$ improves on the bounds in our previously stated theorems via this formulation.
\end{remark}

\section{Conclusion}\label{sec:conclusion}

We have given the first effective upper bounds on $2\cdott d_\gh(S^n,S^k)$ that apply to all possible $n$ and $k$.
We determined the Gromov--Hausdorff distance $2d_\gh(S^1, S^k)$ exactly, and characterized the asymptotic behavior of $\pi-2\cdott d_\gh(S^2,S^k)$.
For general $n$ and $k$ it is likely that our upper bounds are not tight. 
A natural next step would be to try and determine the exact asymptotics of $\pi - 2\cdott d_\gh(S^n,S^k)$ for fixed $n\ge 3$.
\asymptoticconj*

Another natural next step would be to consider asymptotics for spheres with a fixed gap in dimension.
\begin{question}\label{q:codim}
    Fix $m\ge 1$.
    What is the asymptotic behavior of the quantity\[
    \pi - 2\cdott d_{\gh}(S^n, S^{n+m})
    \]
    as $n\to \infty$? 
    What is the limit of this quantity as $n\to \infty$?
\end{question}
\Cref{thm:ntok} gives a bound on this quantity that tends to zero, but for $m=1$ it is known that the true value of the limit is at least $\tfrac{\pi}{3}$ (see \cite[Theorem 1.2]{dghpolymath}).
Thus a full answer to \Cref{q:codim} will likely require new ideas.

The present paper was motivated our collaborative work in \cite{dghpolymath}, in which the geodesic metric played a primary role, and so we have restricted our attention to this setting.
Lim, M\'emoli, and Smith~\cite[Corollary 9.8]{LMS22} showed that $d_\gh(S^n_E, S^k_E) \le \sin(d_\gh(S^n, S^k))$ where the subscript $E$ denotes the Euclidean metric.
Hence we have the following corollaries of \Cref{thm:1tok} and \Cref{thm:ntok}.
\begin{corollary}
 Let $\ell\ge 1$ be any integer. Then \begin{align*}
     d_\gh(S^1_E, S^{2\ell}_E) &\le \sin\left(\frac{2\pi \ell}{2(2\ell +1)}\right)\quad \text{and}\\
        d_\gh(S^1_E, S^{2\ell + 1}_E) &\le \sin\left(\frac{2\pi \ell}{2(2\ell +1)}\right). 
 \end{align*}   
\end{corollary}
\begin{corollary}
    For every $1\le n < k < \infty$, we have \[
    d_\gh(S^n_E, S^k_E) \le \sin \left(\frac{\pi k}{2(k+1)}\right). 
    \]
\end{corollary}

\begin{question}
Are the corollaries above tight in any cases?
What is the distortion of the correspondences we have constructed when considered with the Euclidean metric?
\end{question}

To obtain the best possible results from \Cref{thm:induced-correspondence}, we must construct antipodal point sets in $S^n$ and $S^k$ that have large separation and small Voronoi diameter.
If we only optimize the separation, this is the problem of finding good packings in $\RP^n$ and $\RP^k$ (i.e. projective codes).
If we only optimize the Voronoi diameter, this is the problem of finding a good cover of $\RP^n$ by metric balls. 
Each of these problems has been extensively studied, but as are not aware of any work that seeks to balance both conditions simultaneously.

\begin{question}
Which antipodal point sets $P\subseteq S^n$ minimize the quantity $\max\{\pi-\sep(P), \vdiam(P)\}$?
\end{question}

\section*{Acknowledgements}
We are grateful to Henry Adams, Boris Bukh, Florian Frick, Sunhyuk Lim, and Facundo M{\' e}moli for helpful discussions and feedback.
We are also grateful to all our collaborators in the Gromov--Hausdorff, Borsuk--Ulam, Vietoris--Rips polymath project \cite{dghpolymath} which initiated our work on this project. 

\bibliographystyle{plain}
\bibliography{dgh}

\end{document}